\documentclass[12pt]{amsart}
\usepackage{epsfig,color}

\usepackage{mathrsfs}
\usepackage{amssymb}

\theoremstyle{definition}

\newtheorem{theorem}{Theorem}[section]
\newtheorem{definition}[theorem]{Definition}

\newtheorem{proposition}[theorem]{Proposition}
\newtheorem{lemma}[theorem]{Lemma}
\newtheorem{claim}[theorem]{Claim}
\newtheorem{remark}[theorem]{Remark}
\newtheorem{convention}[theorem]{Convention}
\newtheorem{corollary}[theorem]{Corollary}

\numberwithin{equation}{section}

\usepackage{enumerate}

\newcommand{\K}{\mathcal{K}}
\renewcommand{\L}{\mathcal{L}}

\newcommand{\eps}{\varepsilon}
\newcommand{\R}{\mathbb{R}}
\newcommand{\M}{\mathcal{M}}
\newcommand{\mcfK}{\mathcal{K}}
\newcommand{\mcfM}{\mathcal{M}}
\newcommand{\dilD}{\mathcal{D}}

\DeclareMathOperator*{\dist}{dist}
\DeclareMathOperator*{\graph}{graph}

\usepackage{comment}

\title{Free boundary flow with surgery}

\author{Robert Haslhofer}

\begin{document}
\begin{abstract} In this paper, we prove the existence of mean curvature flow with surgery for mean-convex surfaces with free boundary. To do so, we implement our recent new approach for constructing flows with surgery without a prior estimates in the free boundary setting. The flow either becomes extinct in finite time or for $t\to\infty$ converges smoothly in the one or two sheeted sense to a finite collection of stable connected minimal surfaces with empty or free boundary (in particular, there are no surgeries for $t$ sufficiently large). Our free boundary flow with surgery will be applied in forthcoming work with Ketover, where we will address the existence problem for $3$ free boundary minimal disks in convex balls. 
\end{abstract}
\date{\today}
\maketitle

\tableofcontents

\section{Introduction}

Geometric flows with surgery provide a controlled way of flowing through singularities by cutting along a suitable collection of necks and gluing in standard caps.
Mean curvature flow with surgery was first constructed in the setting of two-convex hypersurfaces in $\mathbb{R}^N$ for $N\geq 4$ in pioneering work by Huisken-Sinestrari \cite{huisken-sinestrari3}, which in particular yielded a topological classification of two-convex hypersurfaces. The case $N=3$ has then been solved independently by Brendle-Huisken \cite{BrendleHuisken} and Kleiner and the author \cite{HK}.
The construction has been later generalized to other ambient manifolds \cite{BrendleHuisken_ambient,HKetover}, low entropy flows \cite{MramorWang,DanielsHolgate}, and higher codimensions \cite{Nguyen_surgery,LangfordNguyen}. Applications include the work on moduli-spaces in \cite{BHH1,BHH2}, the construction of foliations in \cite{HKetover,LiokumovichMaximo}, and the proof of the low entropy Sch\"onflies conjecture in \cite{DanielsHolgate}.\\

However, a problem that until recently seemed inaccessible is the construction of a flow with surgery in the setting of mean-convex free boundary surfaces. The reason for this is that both the approach by Brendle-Huisken \cite{BrendleHuisken} and the approach by Kleiner and the author \cite{HK} crucially rely on the noncollapsing result of Andrews \cite{andrews1}, which is only available in the setting without boundary. In this paper, we solve this problem for mean-convex surfaces with free boundary in any smooth convex domain $D$. More precisely, throughout this paper, $D$ denotes a smooth connected compact domain in a Riemannian three-manifold $(N,g)$, such that $\partial D$ is nonempty and has positive second fundamental form. Our main result gives existence, as well as long-time behaviour, of a flow with surgery with free boundary starting at any mean-convex free boundary initial surface $M_0=\partial K_0\subset D$:

\begin{theorem}
Given any smooth compact strictly mean-convex free boundary domain $K_0\subset D$, there exists a free boundary flow with surgery starting at $K_0$. Moreover, the flow either becomes extinct in finite time or for $t\to\infty$ converges smoothly in the one or two-sheeted sense to a finite collection of stable connected minimal surfaces with empty or free boundary (in particular, there are no surgeries for $t$ sufficiently large).
\end{theorem}

For the purpose of the present paper a free boundary flow with surgery is a free boundary $(\delta,\mathcal{H})$-flow in the domain $D$ as introduced in Definition \ref{def_MCF_surgery}. In particular, $\delta>0$ is a small parameter that captures the quality of the surgery necks and half necks, and $\mathcal{H}$ is a triple of curvature scales $H_{\textrm{trigger}}\gg H_{\textrm{neck}}\gg H_{\textrm{thick}}\gg 1$, which is used to specify more precisely when and how surgeries are performed. Moreover, there is a finite set of surgery times $t_i$, where (i) some necks or half necks in the presurgery domain $K_{t_i}^-$ are replaced by caps or half caps yielding a domain $K_{t_i}^\sharp\subseteq K_{t_i}^-$, 
and/or (ii) some connected components covered entirely by high curvature regions are discarded yielding the postsurgery domain $K_{t_i}^+\subseteq K_{t_i}^\sharp$.\\

As a byproduct of our existence proof we also obtain a canonical neighborhood result:

\begin{theorem}
Given any $\eps>0$, regions of sufficiently large curvature are $\eps$-close (after rescaling to unit curvature) for any sufficiently good surgery parameters to either (i) the evolution of a standard cap preceded by a round shrinking cylinder, or a round shrinking cylinder, round shrinking sphere, translating bowl or ancient oval, or (ii) the evolution of a standard half cap preceded by a round shrinking half cylinder, or a round shrinking half cylinder, round shrinking half sphere, translating half bowl or ancient half oval.
\end{theorem}

\begin{remark}The existence and canonical neighborhood theorem will be applied in forthcoming work with Ketover \cite{HKetover2}, based on the methodology from our earlier paper \cite{HKetover}, to address the existence problem for $3$ free boundary minimal disks in convex balls.
\end{remark}

To prove our results we implement our recent new approach from \cite{H_surgery}, which is based on the theory of weak solutions rather than a priori estimates for smooth solutions, in the free boundary setting. Specifically, 
to begin with, after introducing the necessary definitions, we first prove a quantitative estimates for the space-time Hausdorff distance between the free boundary flow with surgery and the free boundary level set flow.
We then study sequences $\K^j$ of free boundary $(\delta,\mathcal{H}^j)$-flows, with the same mean-convex initial condition $K_0\subset D$, where $\delta\leq\bar{\delta}$ and where the curvature scales $\mathcal{H}^j$ improve along the sequence.
Given any sequence of rescaling factors $\lambda_j\to \infty$, we consider the blowup sequence $\widetilde{\K}^j$,
which is obtained from $\mathcal{K}^j$ by translating $X_j$ to the origin and parabolically rescaling by $\lambda_j$. To keep track of multiplicities we also consider the associated family of Radon measures $\widetilde{\mathcal{M}}^{j}$.

We then establish a hybrid compactness theorem, which allows us to pass to a limits of such generalized blowup sequences, which are smooth near the surgery regions but potentially singular in all other regions. Moreover, using Edelen's monotonicty formula for free boundary flows from \cite{Edelen_brakke} we rule out microscopic surgeries, namely the potential scenario that one has convergence to a plane or half plane with surgeries satisfying $\lambda_j/H_{\textrm{neck}}^j\to 0$.

Next, we generalize the theory of mean-convex Brakke flows with free boundary by Edelen, Ivaki, Zhu and the author \cite{EHIZ} to our setting of hybrid limits of free boundary flows with surgery, and in particular establish multiplicity-one, regularity and convexity. As a consequence of these results, taking also into account the recent classification of ancient solutions from \cite{BC,ADS}, we then establish the canonical neighborhood theorem. 

Using the canonical neighborhood theorem, via a continuity argument similarly as in \cite{HK}, we can establish the existence of free boundary flow with surgery on arbitrarily large time intervals. Finally, borrowing an argument of Brendle-Huisken \cite{BrendleHuisken_ambient}, in combination with the long-time limit results from \cite{EHIZ} and our distance estimate to the free boundary level set flow, we can upgrade this to existence for all times and convergence.

\bigskip

\noindent\textbf{Acknowledgments.}
I thank Jonathan Zhu for helpful discussions. This research has been supported by an NSERC Discovery Grant and a Sloan Research Fellowship.
\bigskip

\section{Definitions and basic properties}

In this section, we define free boundary flows with surgery and establish some basic properties.  Let us begin with the following flexible notion of free boundary $\delta$-flows:

\begin{definition}[free boundary $\delta$-flow]\label{def_alphadelta}
A \emph{free boundary $\delta$-flow} $\K$ in $D$ is a collection of finitely many smooth free boundary strictly mean-convex mean curvature flows $\{K_t^i\subset D\}_{t\in[t_{i-1},t_{i}]}$ ($i=1,\ldots,k$; $t_0<\ldots< t_k$) such that:
\begin{enumerate}[(a)]
\item for each $i=1,\ldots,k-1$, the final time slices of some collection of disjoint strong $\delta$-necks or strong half $\delta$-necks (see Definition \ref{def_strongneck}) are replaced by pairs of standard caps or standard half caps (see Definition \ref{def_replacestd}),
 giving a domain $K^\sharp_{t_{i}}\subseteq K^{i}_{t_{i}}=:K^-_{t_{i}}$.\label{def_delta_1st}\label{item_delta1}
\item the initial time slice of the next flow, $K^{i+1}_{t_{i}}=:K^+_{t_{i}}$, is obtained from $K^\sharp_{t_{i}}$ by discarding some connected components.
\item there exists some $r_\sharp=r_\sharp(\K)>0$, such that all necks and half necks in item \eqref{item_delta1} have radius $r\in[\frac{1}{2}r_\sharp,2 r_\sharp]$.
\end{enumerate}
\end{definition}

The above definition relies on the following two further definitions:

\begin{definition}[strong $\delta$-neck and strong half $\delta$-neck]\label{def_strongneck}
We say that a free boundary $\delta$-flow $\K=\{K_t\subset D\}_{t\in I}$ has a
\begin{itemize}
\item \emph{strong $\delta$-neck} with center $p\in \mathrm{Int}(D)$ and radius $r$ at time $t_0\in I$, if the rescaled flow
$\{r^{-1}(\exp_{p}|_{B_{1/\delta}(0)})^{-1}(K_{t_0+r^2t})\}_{t\in(-1,0]}$ is $\delta$-close in $C^{\lfloor 1/\delta\rfloor}$ in $B_{1/\delta}(0)\times (-1,0]$ to the evolution of a solid round cylinder $\bar{B}^{2}\times \R$ with radius $1$ at $t=0$.
\item \emph{strong half $\delta$-neck} with center $p\in \partial D$ and radius $r$ at time $t_0\in I$, if the rescaled flow
$\{r^{-1}(\Phi_{p}^{-1}|_{B^+_{1/\delta}(0)})^{-1}(K_{t_0+r^2t})\}_{t\in(-1,0]}$ is $\delta$-close in $C^{\lfloor 1/\delta\rfloor}$ in $B^+_{1/\delta}(0)\times (-1,0]$ to the evolution of a solid round half cylinder $\bar{B}_+^{2}\times \R$ with radius $1$ at $t=0$.
\end{itemize}
\end{definition}

Here, we recall that for any $p\in \partial D$ the normal exponential map defines a diffeomorphism $\Phi_p$ from an open neighborhood of $p$ in $D$ to an open half ball $B^+_\rho(0)\subset\mathbb{R}^3_+:=\mathbb{R}_{\geq 0}\times\mathbb{R}^2$.

\begin{definition}[replacement by standard caps or standard half caps]\label{def_replacestd}
Given a cap separation parameter $\Gamma<\infty$ and any $\delta\leq\tfrac{1}{100\Gamma}$, 
we say that the final time slice of a strong $\delta$-neck  with center $p\in \mathrm{Int}(D)$ and radius $r$ \emph{is replaced by a pair of standard caps} or the final time slice of a strong half $\delta$-neck  with center $p\in\partial D$ and radius $r$  \emph{is replaced by a pair of standard half caps}, respectively,
if the presurgery domain $K^-$ is replaced by a postsurgery domain $K^\sharp\subseteq K^-$ with free boundary in $D$, such that:
\begin{enumerate}[(a)]
\item the modification takes places inside $B=B(p,5\Gamma r)$.
 \item there are bounds for the second fundamental form and its derivatives:
$$\sup_{\partial K^\sharp\cap B}|{\nabla^\ell A}|\leq C_\ell r^{-1-\ell}\qquad (\ell=0,1,2,\ldots).$$
 \item for every point $p_\sharp\in \partial K^\sharp\cap B$ with $\lambda_1(p_\sharp)< 0$, there is a point
 $p_{-}\in\partial K^{-}\cap B$ with $$\frac{\lambda_1}{H}(p_{-})\leq\frac{\lambda_1}{H}(p_{\sharp}).$$
 \item the domain $r^{-1}(\exp_{p}|_{B_{10\Gamma r}(0)})^{-1} (K^\sharp)$ is $\delta'(\delta)$-close in $B(0,10\Gamma)$ to a pair of standard caps, or 
 the domain $r^{-1}(\Phi_{p}|_{B_{10\Gamma r}(0)})^{-1} (K^\sharp)$ is $\delta'(\delta)$-close in $B^+(0,10\Gamma)$ to a pair of standard half caps, respectively,
that are at distance $\Gamma$ from the origin. Here, $\delta'(\delta)\to 0$ as $\delta\to 0$.\label{def_surgery_delta_hat}
\end{enumerate}
\end{definition}

Here, a \emph{standard cap} $K^{\textrm{st}}\subset \R^3$  is a smooth convex domain  that coincides with a solid round half-cylinder of radius $1$ outside a ball of radius $2$. Similarly, we call $K^{\textrm{st}}\cap \mathbb{R}^3_+$ a \emph{standard half cap}.
\bigskip

A free boundary flow with surgery is a free boundary $(\delta,\mathcal{H})$-flow defined as follows:

\begin{definition}[free boundary flow with surgery]\label{def_MCF_surgery}
A \emph{free boundary $(\delta,\mathcal{H})$-flow} in $D$, where $\mathcal{H}=(H_{\textrm{thick}},H_{\textrm{neck}},H_{\textrm{trigger}})$, is a free boundary $\delta$-flow $\{K_t\subset D\}_{t\geq 0}$ with smooth free boundary strictly mean-convex initial condition $K_0\subset D$ such that:
\begin{enumerate}[(a)]
\item $H\leq H_{\textrm{trigger}}$ everywhere, and
surgery and/or discarding occurs precisely at times $t$ when $H=H_{\textrm{trigger}}$ somewhere. 
\item The collection of necks and half necks in Definition \ref{def_alphadelta}\eqref{def_delta_1st} is a minimal collection of solid $\delta$-necks and solid half $\delta$-necks of curvature $H_{\textrm{neck}}$ which
separate the set $\{H=H_{\textrm{trigger}}\}$ from $\{H\leq H_{\textrm{thick}}\}$ in $K_t^-$.\label{def_mcf_surgery2}
\item $K_t^+$ is obtained from $K_t^\sharp$ by discarding precisely those connected components with $H>\tfrac{1}{10}H_{\textrm{neck}}$ everywhere. For each pair of facing surgery caps or surgery half caps, precisely one is discarded. 
\item If a strong $\delta$-neck or strong half $\delta$-neck from item \eqref{def_mcf_surgery2} also is a strong $\hat{\delta}$-neck or strong half $\hat{\delta}$-neck for some $\hat{\delta}<\delta$, then property \eqref{def_surgery_delta_hat} of Definition \ref{def_replacestd} also holds with $\hat{\delta}$ instead of $\delta$.
\end{enumerate}
\end{definition}

As a consequence of the definitions we have the following basic properties:

\begin{proposition}[basic properties]\label{basic_prop}
There exist $\bar{\delta}>0$ and $\Gamma_0<\infty$, such that any free boundary $\delta$-flow $\mathcal{K}$ in $D$ with surgery quality $\delta\leq\bar{\delta}$ and cap separation parameter $\Gamma\geq \Gamma_0$ satisfies the following:
\begin{enumerate}[(a)]
\item If $p$ is the center of a surgery neck or surgery half neck of radius $r$, then there are no other surgeries in $B(p,\frac{1}{10}\delta^{-1} r)$.
\item For every ball ${B}$ we have $|{\partial K_{t_1}\cap {B}}|\leq |{\partial K'\cap {B}}|$
for every $K'$ that agrees with $K_{t_1}$ outside ${B}$ and satisfies $K_{t_1}\subseteq K'\subseteq K_{t_0}$ for some $t_0<t_1$.
\end{enumerate}
Moreover,  any free boundary $(\delta,\mathcal{H})$-flow in $D$ for $t\leq T$ satisfies $H>C^{-1}$ and $|A|\leq CH$, where $C=C(D,K_0,T)<\infty$.
\end{proposition}

\begin{proof}
The spatial separation of surgeries follows directly from the definitions. Next, the one-sided minimization follows using the geometric measure theory argument in \cite[Section 3]{White_size}. Finally, the claimed bounds for $H$ and $|A|/H$ follow from  \cite{Edelen_convexity} (which simplifies a lot in our setting thanks to the convexity of $D$), and the definition of surgeries.
\end{proof}

\begin{convention}\label{conv}
We now fix a suitable standard cap $K^{\textrm{st}}$ and cap separation parameter $\Gamma<\infty$. Moreover, throughout this paper $\bar{\delta}>0$ and $\eps>0$ denote sufficiently small constants (by convention, these constants can be decreased finitely many times as needed or convenient).
\end{convention}

\bigskip

\section{Distance to free boundary level set flow} 

In this section, we prove a quantitative estimate for the distance to the free boundary level set flow from \cite{GigaSato}. We will identify free boundary $\delta$-flows with their spacetime track
\begin{equation}
\mathcal{K}=\bigcup_t K_t \times\{t\} \subset D\times \mathbb{R},
\end{equation}
where we set $K_{t_i}:=K_{t_i}^-$ at surgery times to ensure that $\mathcal{K}$ is a closed subset of space-time.

\begin{proposition}[distance to free boundary level set flow]\label{prop_levelset}
Given any $T<\infty$, there exist constants $\bar{\delta}=\bar{\delta}(D,T)>0$ and $C=C(D,K_0,T)<\infty$ such that if $\K$ is a free boundary
$(\delta,\mathcal{H})$-flow ($\delta\leq\bar{\delta}$) with initial condition $K_0$, where $H_{\textrm{neck}}\geq C$, and $\L$ is the free boundary level set flow with the same initial condition $K_0$, then
\begin{equation}
d_{\mathrm{H}}\left(\K\cap \{t\leq T\},\L\cap \{t\leq T\}\right)\leq C H_{\mathrm{neck}}^{-1}\, ,
\end{equation}
where $d_{\mathrm{H}}$ denotes the Hausdorff distance of the spacetime tracks.
\end{proposition}

\begin{proof}
Recall from \cite{EHIZ} that the free boundary level set flow is the maximal family of closed sets starting at $K_0$ that does not bump into any smooth free boundary subsolution of the mean curvature flow.
	Observing that free boundary $(\delta,\mathcal{H})$-flows do not bump into any smooth free boundary subsolution of the mean curvature flow, we thus get $\K\subseteq \L$.\\
	
	To estimate the distance from the other direction, observe that similarly as in \cite[Claim 3.2]{H_surgery}, there exists an $\eta<\infty$, such that if $t$ is a surgery time of $\mathcal{K}$ and $B_r \subseteq K^{-}_t\cap\mathrm{Int}(D)$ is a geodesic ball of radius $r> \eta H_{\mathrm{neck}}^{-1}$, then $B_r \subseteq K^{+}_t$. Also, note that there is a $C_0=C_0(K_0)<\infty$, such that for all $\gamma>0$ small enough there is some $\tau \in [C_0^{-1}\gamma, C_0 \gamma]$, such that \begin{equation}
d(K_\tau,D\setminus K_0)=\gamma.
\end{equation}
Denoting by $-\kappa$ a lower bound for the Ricci curvature, we choose $\gamma=\eta e^{\kappa T}/H_{\textrm{neck}}$, which is allowed provided $H_{\textrm{neck}}$ is large enough. With the corresponding $\tau$, we let $\mathcal{L}^\tau$ be the level set flow with initial condition $L^\tau_0=K_\tau$.

\begin{claim}[evolution of distance]\label{claim_incl}
We have $ L^\tau_t \subseteq K_t$ for all $t
\leq T$ with the estimate
\begin{equation}
d( L^\tau_t, D\setminus K_t)\geq \gamma e^{-\kappa t}.
\end{equation}
\end{claim}

\begin{proof}
Consider the function $f(t):=e^{\kappa t}d( L^\tau_t, D\setminus K_t)$, and note that $f(0)=\gamma$. Now, denoting by $\frac{d}{dt}$ the lim inf of difference quotients and by $H$ the mean curvature in the viscosity sense, away from the surgery times we can estimate
\begin{equation}
\frac{d}{dt}f(t)\geq \kappa e^{\kappa t}d( L^\tau_t, D\setminus K_t) + e^{\kappa t} (H(p)-H(q))\geq 0,
\end{equation}
where we observed that thanks to the strict convexity of $D$ the distance at time $t$ is realized by points $p\in \partial L^\tau_t\cap \mathrm{Int}(D)$ and $q\in \partial K_t\cap \mathrm{Int}(D)$, and these points satisfy $H(p)-H(q)\geq -\kappa d(p,q)$ by mean curvature comparison. Finally, if $t_i$ is a surgery time, then  $d(L^\tau_{t_i},D\setminus K_{t_i}^-) \geq \gamma e^{-\kappa t}$ 
    together with the above observation implies $d(L^\tau_{t_i},D\setminus K_{t_i}^+) \geq \gamma e^{-\kappa t}$.
\end{proof}    
    
Summarizing, we have shown that $\mathcal{L}^\tau \subseteq \mathcal{K} \subseteq \mathcal{L}$,
 where $\tau$ is comparable to $H_{\mathrm{neck}}^{-1}$. Since $\mathcal{L}^\tau$ and $\mathcal{L}$ just differ by a shift by $\tau$ in time direction, this implies the assertion.
\end{proof}

\bigskip

\section{Hybrid compactness theorem for free boundary flows}

Let  $\K^j$ be a sequence of free boundary $(\delta,\mathcal{H}^j)$-flows, where $\delta\leq\bar{\delta}$, with the same strictly mean-convex initial condition $K_0$, and suppose that $\mathcal{H}^j\to\infty$, 
where we use the abbreviation
\begin{equation}
\mathcal{H}^j\to \infty \quad:\Leftrightarrow\quad
\min\left(H^j_{\textrm{thick}} , \frac{  H^j_{\textrm{neck}}}{H^j_{\textrm{thick}} }, \frac{  H^j_{\textrm{trigger}}}{H^j_{\textrm{neck}} }   \right)\to \infty.
\end{equation}
Let $\lambda_j\to \infty$ be a sequence of rescaling factors, and let $X_j=(p_j,t_j)\in\partial \mathcal{K}^j$ be a sequence of space-time points satisfying $ t_j\leq T$ for some $T<\infty$. Consider the blowup sequence
\begin{equation}
\widetilde{\K}^j:=\mathcal{D}_{\lambda_j}( \K^j-X_j),
\end{equation}
which is obtained from $\mathcal{K}^j$ by translating $X_j$ to the origin and parabolically rescaling by $\lambda_j$, and where for ease of notation we pretend that $D$ is a subset of three-dimensional Euclidean space. Moreover, we also consider the associated family of Radon measures
\begin{equation}\label{ass_radon}
\widetilde{\mathcal{M}}^{j}=\big\{  \tilde{\mu}^j_t=\mathcal{H}^2 \lfloor \partial \tilde{K}^j_t\big\},
\end{equation}
where we set $K_t:=K_t^-$ at surgery times. Furthermore, we can assume that 
\begin{equation}
\Lambda:=\lim_{j\to \infty} \frac{\lambda_j}{H_{\textrm{neck}}^j} \in [0,\infty]
\end{equation}
exists. We then say that $(\widetilde{\mathcal{K}}^j,\widetilde{\mathcal{M}}^j)$ is a \emph{$\Lambda$-blowup sequence}. For $\Lambda=0$ we have:

\begin{proposition}[small blowups]\label{prop_L0}
 If $\Lambda=0$, then a subsequence of $\widetilde{\mathcal{K}}^j$ Hausdorff converges to a limit $\mathcal{K}$, which is a blowup limit of the free boundary level set flow of $K_0$. 
\end{proposition}

\begin{proof}Let $\mathcal{L}$ be the free boundary level set flow of $K_0$. By Proposition \ref{prop_levelset} (distance from free boundary level set flow) we have  $d_{\mathrm{H}}(\mathcal{K}^j\cap\{t\leq 2T\},\mathcal{L}\cap\{t\leq 2T\}) \leq C/H^j_{\text{neck}}$, hence
\begin{align}
d_{\mathrm{H}}(\widetilde{\mathcal{K}}^j\cap\{ t\leq \lambda_j^{2}T\},\mathcal{D}_{\lambda_j}(\mathcal{L}-X_j)\cap\{ t\leq \lambda_j^{2}T\}) \leq C\lambda_j/H^j_{\text{neck}}.
\end{align}
Since $\lim_{j \to \infty}\lambda_j/H^j_{\text{neck}}=0$ by assumption, it follows that $\widetilde{\mathcal{K}}^j$ and $\mathcal{D}_{\lambda_j}(\mathcal{L}-X_j)$ converge subsequentially in the Hausdorff sense to the same limit $\K$. This proves the proposition.
\end{proof}

Next, for $0<\Lambda<\infty$ the limiting objects will be either ancient Brakke $\delta$-flows in $\mathbb{R}^3$ as in \cite[Definition 4.2]{H_surgery} or ancient free boundary Brakke $\delta$-flows in $\mathbb{R}^3_+$ defined as follows:

\begin{definition}[{ancient free boundary Brakke $\delta$-flow}]\label{Def_Brakke_surgery}
An \emph{ancient free boundary Brakke $\delta$-flow in $\mathbb{R}^3_+$} is a pair $(\mathcal{K},\mathcal{M})$ consisting of a nested family of closed sets $\mathcal{K}=\{K_t\subset\mathbb{R}^3_+\}_{t \in (-\infty,T)}$ and a family of Radon measures $\mathcal{M}=\{\mu_t\}_{t \in (-\infty,T)}$ in $\mathbb{R}^3_+$, for which
there exists a constant $r_\sharp=r_\sharp(\K,\M)\in (0,\infty)$ and a disjoint collection of two-sided parabolic balls $P_i=B(p_i,50\Gamma r_{\sharp}) \times (t_i-\tfrac12 r_{\sharp}^2,t_i+\eps r_{\sharp}^2)$ such that:
\begin{enumerate}[(a)]
 \item For $t\in (t_i-\tfrac12 r_{\sharp}^2,t_i+\eps r_{\sharp}^2)$, we have $\mu_t\lfloor B(p_i,50\Gamma r_{\sharp})=\mathcal{H}^2 \lfloor \partial K_t\cap B(p_i,50\Gamma r_{\sharp})$, and for $t\neq t_i$ the sets $K_t\cap B(p_i,50\Gamma r_{\sharp})$ are smooth and evolve by free boundary mean curvature flow. At time $t_i$ a strong $\delta$-neck or strong half $\delta$-neck (see Definition \ref{def_strongneck}) of radius $r_i\in [r_\sharp/2,2r_\sharp]$ centered at $(p_i,t_i)$ is replaced by a pair of standard caps or standard half caps (see Definition \ref{def_replacestd}), and possibly some connected components of $K_{t_i}$ are discarded.\label{Def_Brakke_surgery1}
 \item Considering the somewhat smaller $P_i'=B(p_i,25\Gamma r_{\sharp}) \times (t_i-\tfrac14 r_{\sharp}^2,t_i+\tfrac12\eps r_{\sharp}^2)$, we have that $\mathcal{M}$ is a free boundary integral Brakke flow away from $\cup_i P_i'$, and $ \partial \mathcal{K}\setminus \cup_i P_i'=\textrm{spt}\mathcal{M}\setminus \cup_i P_i' $.
\end{enumerate}
\end{definition}

Here, we use the notion of free boundary integral Brakke flows from Edelen \cite{Edelen_brakke}, and in particular recall that the support consists of all space-time points where the (reflected) Gaussian density is at least one.

\begin{remark}[reflection]
In light of \cite[Proposition 4.6]{Edelen_brakke} any ancient free boundary Brakke $\delta$-flow in $\mathbb{R}^3_+$ can be reflected to a flow without boundary in $\mathbb{R}^3$, which satisfies all the axioms of an ancient Brakke $\delta$-flow in $\mathbb{R}^3$ with the caveat that at surgery times the closeness to a standard cap from item \eqref{def_surgery_delta_hat} of Definition \ref{def_replacestd} now only holds in the $C^1$ sense.
\end{remark}

Next, convergence of $\Lambda$-blowup sequences to Brakke $\delta$-flows is defined similarly as in \cite[Definition 4.2]{H_surgery}, and convergence to free boundary Brakke $\delta$-flows is defined as follows:

\begin{definition}[{convergence to free boundary Brakke $\delta$-flow}]\label{Def_brakke_conv_surg}
Given $0<\Lambda<\infty$, a $\Lambda$-blowup sequence $(\widetilde{\K}^j,\widetilde{\M}^j)$  \emph{converges to a free boundary Brakke $\delta$-flow $(\mathcal{K},\mathcal{M})$} if:
 \begin{enumerate}[(a)]
 \item The space-time tracks $\widetilde{\mathcal{K}}^j$ Hausdorff converge to the space-time track $\mathcal{K}$.
  \item If $P_i=B(p_i,50\Gamma r_{\sharp}) \times (t_i-\tfrac12 r_{\sharp}^2,t_i+\eps r_{\sharp}^2)$ is a surgery region of $(\mathcal{K},\mathcal{M})$ as in Definition \ref{Def_Brakke_surgery} (ancient free boundary Brakke $\delta$-flow), then for some $(p_i^j,t_i^j)\in \widetilde{\mathcal{K}}^j$ converging to $(p_i,t_i)$ the forwards and backwards portion $\{\tilde{K}^j_{t+t_i^j}-p^j_i\}_{t \in (0,\eps r_{\sharp}^2)}$ and $\{\tilde{K}^j_{t+t_i^j}-p^j_i\}_{t \in (-\frac12 r_{\sharp}^2,0]}$ converge smoothly to $\{K_{t+t_i}-p_i\}_{t \in (0,\eps r_{\sharp}^2)}$ and $\{K_{t+t_i}-p_i\}_{t \in (-\frac12 r_{\sharp}^2,0]}$, respectively, in $B(0,50\Gamma r_{\sharp})$.\label{Def_brakke_conv_surg2}
      \item $\widetilde{\M}^j$ converges to ${\M}$ in the sense of free boundary Brakke flows away from $\cup_i P_i'$.
 \end{enumerate}
\end{definition}

\begin{theorem}[hybrid compactness]\label{prop_compactness}
Any $\Lambda$-blowup sequence $(\widetilde{\K}^j,\widetilde{\M}^j)$ with $\Lambda \in (0,\infty)$ has a subsequence that converges to a limit $(\K,\M)$ that is either an ancient Brakke $\delta$-flow in $\mathbb{R}^3$ or an ancient free boundary Brakke $\delta$-flow in $\mathbb{R}^3_+$.
\end{theorem}

\begin{proof} If $\lambda_j d(p_j,\partial D)\to \infty$, then the proof of \cite[Theorem 4.4]{H_surgery} applies, yielding convergence to an ancient Brakke $\delta$-flow in $\mathbb{R}^3$. Hence, after shifting the base-points by controlled rescaled distance we can assume from now on that $p_j\in \partial D$. We can also assume that $\widetilde{\mathcal{K}}^j$ Hausdorff converges to a limit $\mathcal{K}=\{K_t\subset\mathbb{R}^3_+\}_{t \in (-\infty,T)}$, which is a nested family of closed sets.

Since $\Lambda>0$, for any $R<\infty$ the number $N_R^j$ of surgery centers of  $\widetilde{\mathcal{K}}^j$ in the two-sided parabolic ball $P(0,R)=B(0,R)\times (-R^2,R^2)$ is uniformly bounded. After passing to a subsequence, we can assume that $N_R^j=N_R$ is independent of $j$. Moreover, denoting by $\{(p^j_i,t^j_i)\}_{i=1}^{N_R}$ the surgery centers of  $\widetilde{\mathcal{K}}^j$ in $P(0,R)$, and denoting by $r^j_i \in [\Lambda_j/2,2\Lambda_j]$ the neck radii, where $\Lambda_j := \lambda_j/H^j_{\text{neck}}$, after passing to a further subsequence we can assume that
\begin{align}\label{limit_data}
(p^j_i,t^j_i) \to (p_i,t_i)\qquad \textrm{and} \qquad
 r^j_i \to r_i \in [\Lambda/2,2\Lambda].
\end{align}
Set $r_\sharp:=\Lambda$, and recall that $\eps>0$ is a small fixed constant by Convention \ref{conv}. Arguing similarly as in the proof of \cite[Claim 4.5]{H_surgery}, where we now use the local regularity theorem for free boundary flows from \cite{Edelen_brakke}, for each $P_i=B(p_i,50\Gamma r_{\sharp}) \times (t_i-\tfrac12 r_{\sharp}^2,t_i+\eps r_{\sharp}^2)$ we get
\begin{equation}\label{curv_bounds_par}
\limsup_{j\to\infty} \sup_{ \partial \widetilde{\mathcal{K}}^j\cap P_i}|\nabla^\ell A|<\infty\qquad (\ell=0,1,2,\ldots).
\end{equation}
Hence, the convergence in the two-sided parabolic balls $P_i$ is smooth in the sense of Definition \ref{Def_brakke_conv_surg}(\ref{Def_brakke_conv_surg2}). Moreover, away from $\cup_i P_i'$ we can pass to a subsequential weak limit using the compactness theorem for free boundary integral Brakke flows from \cite{Edelen_brakke}. In the surgery regions we define our measure by declaring that $\mu_{t}(A):=\mathcal{H}^2( \partial K_{t} \cap A)$ for $A\subseteq B(p_i,50\Gamma r_{\sharp})$ and $t\in (t_i-\tfrac12 r_{\sharp}^2,t_i+\eps r_{\sharp}^2)$, and observe that this is consistent in the overlap regions.

We now consider a sequence $R_j \to \infty$, and pass to a diagonal subsequence of the above to obtain a global limit $(\mathcal{K},\mathcal{M})$. Observe that our limit satisfies all the properties listed in Definition \ref{Def_Brakke_surgery}\eqref{Def_Brakke_surgery1}, and that $\mathcal{M}$ is a free boundary integral Brakke flow away from $\cup_i P_i'$. Finally, arguing similarly as in the last paragraph of the proof \cite[Theorem 4.4]{H_surgery}, where we now use the (reflected) Gaussian density from \cite{Edelen_brakke}, we see that the support of $\mathcal{M}$ agrees with $\partial \mathcal{K}$ away from $\cup_i P_i'$. This finishes the proof of the theorem.
\end{proof}

Finally, let us deal with the case $\Lambda=\infty$:

\begin{proposition}[large blowups]\label{prop_LI}
 If $\Lambda=\infty$, then $(\widetilde{\K}^j,\widetilde{\M}^j)$ subsequentially converges either (a) in the Hausdorff and Brakke sense to a mean-convex flow in $\mathbb{R}^3$ or a mean-convex free boundary flow in $\mathbb{R}^3_+$ or (b) smoothly to a static or backwards/forwards quasistatic multiplicity-one plane in $\mathbb{R}^3$ or multiplicity-one free boundary half plane in $\mathbb{R}^3_+$.
 \end{proposition}

\begin{proof}
If for every $R<\infty$ the two-sided parabolic ball $P(0,R)$ does not contain points modified by surgeries for infinitely many $j$, then the conclusion (a) holds. Assume now that there is some $R<\infty$, such that $P(0,R)$ contains points modified by by surgeries for all $j$. Since $\Lambda=\infty$, arguing similarly as in the proof of \eqref{curv_bounds_par} we see that for each $\rho<\infty$ we have
\begin{equation}
\limsup_{j\to\infty} \sup_{ \partial \widetilde{\mathcal{K}}^j\cap P(0,\rho)}|\nabla^\ell A|=0\qquad (\ell=0,1,2,\ldots).
\end{equation}
It follows that conclusion (b) holds. This finishes the proof of the proposition.
\end{proof}

A large portion of this paper will deal with analyzing the limits constructed above:

\begin{definition}[generalized limit flow]\label{def_genlim}
A \emph{$\Lambda$-generalized limit flow} is any limit $(\K,\M)$ provided by Theorem \ref{prop_compactness} (hybrid compactness) or Proposition \ref{prop_LI} (large blowups).
\end{definition}

To conclude this section, let us observe that generalized limit flows inherit the one-sided minimization property in the following sense:

\begin{corollary}[one-sided minimization for generalized limit flows]\label{prop_onesidedmin}
Let $(\mathcal{K},\mathcal{M})$ be a $\Lambda$-generalized limit flow (see Definition \ref{def_genlim}). Suppose $\gamma$ is a $1$-cycle that at some time $t$ bounds a $2$-chain in $\partial K_t$, and $\mathcal{H}^2(\mathrm{sing}\, \partial K_t)=0$.
Then, there exists a $2$-chain $\Sigma$ supported in $K_t$ with $\partial \Sigma = \gamma$, such that $\mathcal{H}^2(\Sigma) \leq \mathcal{H}^2(\Sigma')$ for any $2$-chain $\Sigma'$ bounded by $\gamma$.
\end{corollary}

\begin{proof} This follows from the one-sided minimization for smooth flows (see Proposition \ref{basic_prop}), via the same argument as in the the proof of \cite[Theorem 6.1]{White_size}.
\end{proof}

\bigskip

\section{Excluding microscopic surgeries}

The goal of this section is to prove the following theorem.

\begin{theorem}[no microscopic surgeries]\label{thm_no_microscopic}
Suppose $\widetilde{\K}^j$ is a $0$-blowup sequence (with $\delta\leq\bar{\delta}$ small enough) that Hausdorff converges in space-time to a (quasi-)static multiplicity-one plane or half plane. Then, for any $R<\infty$ there exists a $j_0=j_0(R)<\infty$, such that for all $j\geq j_0$ the two-sided parabolic ball $P(0,R)$ contains no points modified by surgeries.
\end{theorem}

\begin{proof} Observing that the argument from the proof of \cite[Theorem 4.4]{H_surgery} already rules out microscopic surgeries in the interior, it suffices to rule out microscopic surgeries at the boundary. Specifically, pretending for ease of notation that $D\subset\mathbb{R}^3_+$ with $0\in\partial D$, suppose towards a contradiction that there is a $0$-blowup sequence $\widetilde{\K}^j$ that Hausdorff converges to $\{x_1\geq 0\}\cap \{x_3\geq 0\} \times (-\infty,0]$, but such that there are surgeries at half necks of radius $r_j\to 0$ centered at $(0,0)$. Set $t_0^j=r_j^2/2$, and instead of Huisken's quantity $\Theta$ from \cite{Huisken_monotonicity} consider Edelen's quantity $\Theta_R$ from \cite[Definition 5.1.1]{Edelen_brakke}, centered at $X_0^j=(0,t_0^j)$, namely
\begin{equation}
 \Theta_R\big(\widetilde{\mathcal{M}}^j,X_0^j,\tau\big)=\int_{\partial\tilde{K}^j_{t_0^j-\tau}} \left( \theta(x,\tau)+\theta(R_jx,\tau)\right)\, dA(x),
\end{equation}
where $\theta$ denotes the truncated Gaussian kernel defined by
\begin{equation}
\theta(x,\tau)=\frac{1}{4\pi \tau}\exp\left(-\frac{|x|^2}{4\tau}\right)\left( 1 -\left(\frac{\rho^2}{\tau}\right)^{3/4} \frac{|x|^2-\alpha\tau}{\rho^2} \right)_{+}^4,
\end{equation}
for suitable choice of $\rho>0$ and $\alpha>0$, and $R_jx$ denotes the reflection of $x$ across $\partial(\lambda_j D)$.
For small backwards time, say $\tau_j=r_j^2$, we are $\delta$-close to a half neck, and thus get
\begin{equation}\label{lowerdensity}
\liminf_{j\to\infty} \Theta_R\big(\widetilde{\mathcal{M}}^j,X_0^j,r_j^2\big)>3/2.
\end{equation}
Now, by \cite[Theorem 5.5]{Edelen_brakke} the function $\tau\mapsto \Theta_R\big(\widetilde{\mathcal{M}}^j,X_0^j,\tau\big)$ is almost monotone if there are no surgeries. Note also that discarding connected components has the good sign. Finally, arguing similarly as in the proof of \cite[Claim 2.17]{HK} we see that the cumulative error in Edelen's monotonicity inequality due to surgeries between $\tau=r_j^2$ and $\tau=\eps\rho^2$ is less than $1/100$, provided $j$ is sufficiently large.
Hence, closeness to a multiplicity-one half plane at scale $\tau=\eps\rho^2$ for $j$ large enough gives the desired contradiction with \eqref{lowerdensity}. 
\end{proof}

\bigskip

\section{Multiplicity-one for free boundary flows}

In this section, we prove that every generalized limit flow  has multiplicity-one. To this end, we will adapt the arguments from \cite{EHIZ} to our setting of flows with surgeries.

\subsection{Large blowups with entropy at most two}

In this subsection, fixing the initial domain $K_0\subset D$, we consider the class $\mathcal{C}$ of all blowup limits $(\mathcal{K},\mathcal{M})$ given by case (a) of Proposition \ref{prop_LI} (large blowups), such that
\begin{enumerate}
\item $\lim_{\tau\to\infty}\Theta(\mathcal{M},0,\tau)\leq 2$ respectively  $\lim_{\tau\to\infty}\Theta(\widetilde{\mathcal{M}},0,\tau)\leq 2$,
\item $(\mathcal{K},\mathcal{M})$ respectively $(\widetilde{\mathcal{K}},\widetilde{\mathcal{M}})$ is not a static or quasistatic multiplicity-two plane.
\end{enumerate}

Here, in case $(\mathcal{K},\mathcal{M})$ is defined in $\mathbb{R}^3_+$ we denote by $(\widetilde{\mathcal{K}},\widetilde{\mathcal{M}})$ the reflected flow in $\mathbb{R}^3$. 

\begin{proposition}[partial regularity]
\label{prop:partial-reg}
For $(\mathcal{K},\mathcal{M})\in\mathcal{C}$ tangent flows at singular points cannot be static or quasi-static. In particular, the dimension of the singular set is at most $1$.
\end{proposition}

Here, the singular set is defined as the collection of all space-time points that do not have a backwards parabolic neighborhood in which the flow is smooth with multiplicity-one.

\begin{proof}
In case $(\mathcal{K},\mathcal{M})$ is defined in $\mathbb{R}^3_+$ we consider the reflected flow $(\widetilde{\mathcal{K}},\widetilde{\mathcal{M}})$.
By the equality case of Huisken's monotonicity formula \cite{Huisken_monotonicity} and the definition of the class $\mathcal{C}$ no tangent flow can be a static or quasistatic plane of higher multiplicity. Together with standard stratification \cite{White_stratification}, remembering that we are working in $\mathbb{R}^3$, the assertion follows.
\end{proof}

\begin{corollary}[static or quasistatic limits]\label{cor:staticlimits}
If $(\mathcal{K},\mathcal{M})\in \mathcal{C}$ is static or quasi-static, then one of the following five cases occurs:
\begin{enumerate}
\item $\mcfK$ is a (quasi-)static half space in $\mathbb{R}^{3}$, and $\mcfM$ is the (quasi-)static plane $\partial \mcfK$.
\item $\mcfM$ is a pair of two (quasi-)static parallel multiplicity-one planes in $\mathbb{R}^{3}$ and $\mcfK$ is the region in between.
\item $\mcfK$ is a (quasi-)static quarter space in $\mathbb{R}^3_+$, and $\mcfM$ is the (quasi-)static half plane $\partial \mcfK$ with multiplicity-one.
\item $\mcfM$ is a pair of(quasi-)static multiplicity-one half planes in $\mathbb{R}^3_+$ with free boundary and $\mcfK$ is the region in between.
\item $\mcfM$ is a (quasi-)static multiplicity one plane in $\mathbb{R}^3_+$ parallel to $\partial\mathbb{R}^3_+$, and $\mcfK$ is the region in between.
\end{enumerate}
\end{corollary}

\begin{proof}
Since it is (quasi-)static, $\partial\mathcal{K}$ must be smooth and flat by Proposition \ref{prop:partial-reg} (partial regularity), and hence a union of one or two disjoint (quasi-)static multiplicity-one planes or half planes. The result then follows from one-sided minimization (Corollary \ref{prop_onesidedmin}).
\end{proof}

\begin{theorem}[separation theorem]\label{thm:sep}
Let $(\mathcal{K},\mathcal{M})\in \mathcal{C}$. In case the flow is defined in $\mathbb{R}^3_+$, suppose that there is a half plane $H$ perpendicular to $\partial\mathbb{R}^3_+$, such that
$H\subseteq \bigcap_t K_t$,
and suppose the complement of $\cap_t K_t$ contains points on each side of $H$. Then, $(\mathcal{K},\mathcal{M})$ is static, and $K_t$ is the region between two parallel half planes perpendicular to $\partial\mathbb{R}^3_+$. Moreover, a similar statement holds in case the flow is defined in entire space and contains a plane.
\end{theorem}

\begin{proof}
Using Proposition \ref{prop:partial-reg} (partial regularity) and  Corollary \ref{prop_onesidedmin} (one-sided minimization for generalized limit flows) we can follow the proof of \cite[Theorem 6.4]{EHIZ} to show that the reflected flow $\widetilde{\mathcal{M}}$ splits into two components $\widetilde{\mcfM}_\pm$, each contained in the respective halfspace defined by $\widetilde{H}$, which is obtained from $H$ by reflection.
Since each Brakke flow $\widetilde{\mcfM}_\pm$ has density at least 1, but the sum of densities is at most 2, each $\widetilde{\mcfM}_\pm$ must have density exactly 1 and hence be a multiplicity-one plane.
Finally, since $\widetilde{K}_t \supseteq \widetilde{H}$ is in particular nonempty for all $t$, we conclude that each plane $\widetilde{\mcfM}_\pm$ is static. This implies the assertion.
\end{proof}

Now, as in \cite[Section 4]{White_size}, for a set $S\subseteq D$, a point $x\in D$, and a radius $r>0$, the \emph{relative thickness} of $S$ in $B(x,r)$ is defined by
\begin{equation}
\mathrm{Th}(S,x,r) =\frac{1}{r} \inf_{|v|=1}  \sup_{y\in S\cap B(x,r)} |\langle v, y-x\rangle|.
\end{equation}

\begin{theorem}[Bernstein-type theorem]\label{thm:bern}
There exists an $\eps>0$ with the following significance. If $(\mcfK,\mcfM)\in\mathcal{C}$ is defined in $\mathbb{R}^3_+$ and there is a point $x$ such that
\begin{equation}\label{eq_ass_bern}
\limsup_{r\to\infty} \mathrm{Th}(K_{-r^2},x,r) < \eps
 \qquad  \mathrm{and}  \qquad 
\liminf_{r\to\infty} \frac{\dist(K_{r^2},x)}{r} <1,
\end{equation}
then $\mcfM$ is either a pair of static parallel multiplicity-one half planes with free boundary or a static multiplicity-one plane parallel to $\partial\mathbb{R}^3_+$. In either case $\mcfK$ is the region in between the planes of $\mcfM$ and $\partial\mathbb{R}^3_+$. Similarly, if $(\mcfK,\mcfM)$ is defined in $\mathbb{R}^{3}$, then under the same assumptions $\mcfM$ is a pair of static parallel multiplicity-one planes, with $\mcfK$ the region in between.
\end{theorem}

\begin{proof}
The statement for flows in $\mathbb{R}^{3}$ follows from the proof of \cite[Theorem 6.4]{H_surgery}, so we focus on the case of free boundary flows in $\mathbb{R}^3_+$.

Since $\mathcal{K}$ is nested and does not bump into any smooth subsolution of the free boundary mean curvature flow, choosing $\eps$ small enough we can apply the expanding hole result from \cite[Corollary 6.7]{EHIZ} to infer that $\Sigma:=\bigcap_t K_t$ is nonempty. Now, consider the flows obtained by translating $(\mathcal{K},\mathcal{M})$ by $(0,-T)$ and let $(\mathcal{K}',\mathcal{M}')$ be a limit as $T\to\infty$. Then $(\mathcal{K}',\mathcal{M}')$ is a static flow with $K_t'=\Sigma$ at any time $t$.
Note that $(\mathcal{K}',\mathcal{M}')$ either is a multiplicity-two half plane, or belongs to the class $\mathcal{C}$ and hence by Corollary \ref{cor:staticlimits} (static or quasistatic limits) and assumption \eqref{eq_ass_bern} is either the region in between a pair of multiplicity-one free boundary half planes or the region bounded by $\partial\mathbb{R}^3_+$ and a parallel multiplicity-one plane. Hence, Theorem \ref{thm:sep} (separation theorem), applied directly in the former cases and applied to the reflected flow in the final case, respectively, implies the result.
\end{proof}

\subsection{Sheeting theorem for $\Lambda$-blowup sequences}
In this subsection, we establish a sheeting theorem for $\Lambda$-blowup sequences with $\Lambda>0$. We start with the following lemma:

\begin{lemma}[slab rescaling]\label{lem:sheetblow}
Let $(\widetilde{\mcfK}^j,\widetilde{\mcfM}^j)$ be a $\Lambda$-blowup sequence with $\Lambda>0$, and suppose that
\begin{equation}\label{eq_ass}
d_{\mathrm{H}}\big(\widetilde{\mcfK}^j\cap P(0,2), (V\times \mathbb{R})\cap P(0,2)\big) \rightarrow 0,
\end{equation}
where $V$ is either a plane or a half plane. Then, there exists a sequence $\mu_j\to\infty$, such that $\mathcal{D}_{\mu_j} \widetilde{\mcfM}^j$ converges smoothly to either (a)
a pair of parallel planes in $\mathbb{R}^{3}$ or (b) a pair of parallel half planes with free boundary in $\mathbb{R}^{3}_+$ or (c) a multiplicity-one plane parallel to $\partial\mathbb{R}^{3}_+$.
Moreover, in all cases $\mathcal{D}_{\mu_j} \widetilde{\mcfK}^j$ converges to the enclosed region.
\end{lemma}

\begin{proof}
Fixing $\eps>0$ small enough, let $\mu_j<\infty$ be the largest number such that for all $r\in [\mu_j^{-1},1]$ we have
\begin{equation}\label{eq_how_thick}
\textrm{Th}(\tilde{K}^j_{-r^2},0,r)\leq \varepsilon\qquad \textrm{and} \qquad d(\tilde{K}^j_{r^2},0)\leq r.
\end{equation}
Assumption \eqref{eq_ass} implies that $\mu_j\to \infty$, so remembering also that $\Lambda>0$ we in particular get
${\mu_j \lambda_j / H^j_{\textrm{neck}}} \to \infty$, and thus
by Proposition \ref{prop_LI} (large blowups) we can take a subsequential limit  of $\mathcal{D}_{\mu_j}(\widetilde{\mathcal{K}}^j, \widetilde{\mathcal{M}}^j)$. By construction, any such limit $(\mathcal{K},\mathcal{M})$ satisfies
\begin{equation}\label{eq_how_thick_lim}
\textrm{Th}(K_{-r^2},0,r)\leq \varepsilon\qquad \textrm{and} \qquad d({K}_{r^2},0)\leq r \qquad \textrm{for all } r\geq 1,
\end{equation}
with at least one inequality being non-strict for $r=1$. In particular, we must be in case (a) of  Proposition \ref{prop_LI}.
Taking also into account Corollary \ref{prop_onesidedmin} (one-sided minimization for generalized limit flows) we thus infer that $(\mathcal{K},\mathcal{M})\in\mathcal{C}$ . Hence,  Theorem \ref{thm:bern}  (Bernstein-type theorem) and the local regularity theorem \cite{white_regularity,Edelen_brakke} imply the assertion.
\end{proof}

Denote by $D^j=\lambda_j D$ the domain of the rescaled flow $(\widetilde{\mcfK}^j,\widetilde{\mcfM}^j)$. To construct a separating surface in case ($a$) and ($b$) of the above lemma, we let $S^j_t$ be the set of centers of open balls $B$, such that $B\cap D^j \subseteq K^j_t$ and $\bar{B}\cap D^j$ touches $\partial K^j_t$ at two or more points. Set
\begin{equation}
\mathcal{S}^j = \bigcup_t S^j_t \cap D^j.
\end{equation}

\begin{theorem}[sheeting theorem]\label{thm_sheeting}
Let $(\widetilde{\mathcal{K}}^j, \widetilde{\mathcal{M}}^j)$ be a $\Lambda$-blowup sequence with $\Lambda>0$, and suppose that
\begin{equation}\label{eq_ass_pl2}
d_{\mathrm{H}}\big(\mcfK^j\cap P(0,4), (V\times \mathbb{R})\cap P(0,4)\big) \rightarrow 0,
\end{equation}
where either $D^j \to \mathbb{R}^{3}$ and $V$ is a plane, or $ D^j \to \mathbb{R}^3_+$ and $V$ is a free boundary half plane.
Then, for all $j$ large, $\mathcal{S}^j \cap P(0,1)$ is a $C^1$-hypersurface that divides $\partial \mathcal{K}^j$ into two nonempty components.
\end{theorem}

\begin{proof}
Observe that thanks to \eqref{eq_ass_pl2} and $\Lambda>0$, for $j$ large enough there are no points modified by surgery in $P(0,3)$. Hence, following the proof of \cite[Theorem 6.10]{EHIZ}, where we now use Lemma \ref{lem:sheetblow} (slab rescaling) in lieu of \cite[Lemma 6.9]{EHIZ}, yields the assertion.
\end{proof}

\subsection{Ruling out generalized limit flows with density two}
As above, fixing the initial condition $K_0\subset D$, we consider all generalized limit flows $(\mathcal{K},\mathcal{M})$ as in Definition \ref{def_genlim}. We recall that they arise as limits of $\Lambda$-blowup sequences, where $\Lambda\in(0,\infty]$.

Given any closed subset $\mathcal{K}'\subset\mathbb{R}^3\times\mathbb{R}$, similarly as in \cite[Section 9]{White_size} we denote by $\phi(\mathcal{K}')$ the infimum of $s>0$ such that
\begin{equation}
d_{\mathrm{H}}\left(\mathcal{K}'\cap P_{-}(0,1/s),(V\times\mathbb{R})\cap P_{-}(0,1/s)\right) < s
\end{equation}
for some plane $V\subset\mathbb{R}^3$ through the origin, and similarly as in \cite[Section 6]{EHIZ} we denote by $\phi_+(\mathcal{K}')$ the infimum of $s>0$ such that
\begin{equation}
d_{\mathrm{H}}\left(\mathcal{K}'\cap P_{-}(0,1/s),(H\times\mathbb{R})\cap P_{-}(0,1/s)\right) < s
\end{equation}
for some half plane $H\subset\mathbb{R}^3_{+}$ that meets $\partial\mathbb{R}^3_{+}$ orthogonally at the origin.\\

\begin{lemma}[isolation]\label{lemma_isolation}
There exists an $\eps>0$, such that if $(\K',\M')$ is a tangent flow to a generalized limit flow at $X=(0,t)$, then we have:
\begin{enumerate}
\item If $0$ is a boundary point and $\phi_+(\mcfK')<\eps$, then $\phi_+(\mcfK')=0$.\label{isol1}
\item If $0$ is a boundary point and $\phi(\mcfK')<\eps$, then $\phi(\mcfK')=0$.\label{isol2}
\item If $0$ is an interior point and $\phi(\mcfK')<\eps$, then $\phi(\mcfK')=0$.\label{isol3}
\end{enumerate}
\end{lemma}
\begin{proof}
Observing that the argument from the proof of \cite[Lemma 6.8]{H_surgery} already shows (\ref{isol3}), it suffices to show (\ref{isol1}) and (\ref{isol2}). Specifically, given any sequence of tangent flows $(\K^j, \M^j)$ at $X_j=(0,t_j)$ to $\Lambda_j$-generalized limit flows in $\mathbb{R}_3^+$, where $\Lambda_j>0$, with ${\phi}_+(\mcfK^j)\rightarrow 0$ or ${\phi}(\mcfK^j)\rightarrow 0$, we must show that for large $j$ we have $\phi_+(\mcfK^j)=0$ or ${\phi}(\mcfK^j)= 0$, respectively.

Note that in particular each $(\K^j, \M^j)$ is an $\infty$-generalised limit flow. Hence, in case (\ref{isol1}) as a consequence of Theorem \ref{thm_sheeting} (sheeting theorem), similarly as in \cite[Lemma 6.16]{EHIZ}, for large $j$ there are functions $f_j$ and $g_j $, defined on an exhaustion of $\mathbb{R}^2_+\times (-\infty,0)$, such that:
\begin{enumerate}
\item Either $f_j<g_j$ everywhere, or $f_j \equiv g_j$.
\item For any $U\subset\subset \mathbb{R}^3_+$ and $[a,b]\subset (-\infty,0)$, for $j$ large enough the region $K^j_t$ coincides in $U$ with the region between $\graph (f_i)$ and $\graph (g_i)$ for all $t\in [a,b]$.
\item $f_j$ and $g_j$ converge smoothly on compact subsets to $0$.
\item $f_j$ and $g_j$ solve the graphical (free boundary) mean curvature flow equation.
\item $f_j$ and $g_j$ are nondecreasing and nonincreasing in time, respectively.
\end{enumerate}
A similar statement holds in case (\ref{isol2}), where the functions $f_j$ and $g_j$ are now defined on an exhaustion of $\partial \mathbb{R}^3_+\times (-\infty,0)$ and we can simply take $f_j\equiv 0$.

Moreover, since $\Lambda_j>0$, by monotonicity each tangent flow is backwardly self-similar, so 
\begin{equation}\label{eq:self-sim-fn}
f_j(rx,r^2t) = rf_j(x,t), \qquad g_j(rx,r^2t)=rg_j(x,t).
\end{equation} 

Now, if $f_j<g_j$ for infinitely many $j$, then using the Harnack inequality similarly as in \cite[Case 1 in the proof of Theorem 9.1]{White_size}, where in case (\ref{isol1}) we consider the reflected functions on $\mathbb{R}^2\times (-\infty,0)$, we can find $c_j>0$, such that a subsequence of $c_j(g_j-f_j)$ converges smoothly on compact subsets to the constant function $u\equiv 1$. However, using (\ref{eq:self-sim-fn}) we infer that $u(rx,r^2t) = ru(x,t)$, which is absurd.

Thus, $f_j\equiv g_j$ for large $j$. But then the functions are constant in $t$, and together with the self-similarity we infer that $f_j\equiv g_j$ is 1-homogenous. Remembering smoothness this implies linearity. Taking also into account the free boundary condition in case (1), we conclude that $\phi_+(\mcfK^j)=0$ or ${\phi}(\mcfK^j)= 0$, respectively, for $j$ large.
\end{proof}

\begin{lemma}[minimal surface]\label{lemma_min_surf}
Let $(\mathcal{K},\mathcal{M})$ be a generalized limit flow. Suppose $(\mathcal{K}',\mathcal{M}')$ is a tangent flow of $(\mathcal{K},\mathcal{M})$ taken taken at a density-two point $X=(x,t)$. Then:
\begin{enumerate}
\item If $x$ is a boundary point, and $\mcfK'$ is a static or quasistatic half plane, then there exist an open neighborhood $U$ of $x$ in $\mathbb{R}^3_+$, an open interval $(a,b)$, and a properly embedded smooth free boundary minimal surface $\Sigma$ in $U$, such that $K_\tau \cap U = \Sigma$ for all $\tau\in(a,b)$.\label{mins1}
\item If $x$ is a boundary point, and $\mcfK'$ is a static or quasistatic plane, then there exists a neighborhood $U$ of 0 in $\mathbb{R}^3_+$, and an open interval $(a,b)$, such that $K_\tau \cap U = \partial \mathbb{R}^3_+\cap U$ for all $\tau\in (a,b)$.\label{mins2}
\item If $x$ is an interior point, and $\mcfK'$ is a static or quasistatic plane, then there is an open neighborhood $U$ of $x$ in $\mathbb{R}^{3}$, an open interval $(a,b)$, and a properly embedded smooth minimal surface $\Sigma$ in $U$, such that $K_\tau \cap U = \Sigma$ for all $\tau\in(a,b)$.\label{mins3}
\end{enumerate}
Furthermore, in the cases (\ref{mins1}) and (\ref{mins2}) we have $\Theta(\widetilde{\mathcal{M}},\cdot)\geq 2$ on all of $\widetilde{\Sigma}\times (-\infty,b]$ or $(\partial \mathbb{R}^3_+\cap U)\times (-\infty,b]$, respectively, and in case (\ref{mins3}) we have $\Theta(\mathcal{M},\cdot)\geq 2$ on all of $\Sigma\times (-\infty,b]$.
\end{lemma}

\begin{proof} 
Note that there are no surgeries in a spacetime neighborhood of $X$. Hence, the assertions (\ref{mins1}), (\ref{mins2}) and (\ref{mins3}) follow from Lemma \ref{lemma_isolation} (isolation) and Theorem \ref{thm_sheeting} (sheeting theorem) exactly as in the proof of \cite[Lemma 6.22]{EHIZ}. 

For the density, in case  (\ref{mins1}) let $\mathcal{Z}$ be the set of points of $\widetilde{\M}$ at which none of the tangent flows is planar (which is necessarily in the complement of the surgery region). 
By general stratification results \cite{White_stratification},  the parabolic Hausdorff dimension of $\mathcal{Z}$ is at most $1$. In particular, the spatial projection $\pi(\mathcal{Z})$ has Hausdorff dimension at most $1$. So by upper semicontinuity of the density, it is enough to show that $\Theta(\widetilde{\M}, (y,t)) =2$ for all $y\in \widetilde{\Sigma}\setminus \pi(\mathcal{Z})$ and $t<b$. To this end, fix $y\in\widetilde{\Sigma}\setminus \pi(\mathcal{Z})$, and let $T^*=\sup\{\tau<b\, |\, \Theta(\widetilde{\M},(y,\tau))\neq 2\}$. Clearly $T^*\leq a$. Suppose towards a contradiction that $T^*>-\infty$. Note that a neighborhood of $(y,T^*)$ is unmodified by surgeries, since otherwise the density near $(y,T^\ast)$ would be less than $2$. Now, consider a tangent flow $(\mathcal{K}' , \mathcal{M}')$ at $(y,T^*)$. If it is a static multiplicity 2 plane, then applying the first part of the theorem shows that the density is 2 for times close to $T^*$, which contradicts the definition of $T^*$. 
If $(\mathcal{K}' , \mathcal{M}')$ was a quasistatic plane or a static plane of multiplicity 1, then we would obtain a contradiction with the fact that $\Theta(\widetilde{\M}, (y,\tau))=2$ for $\tau \in [T^*, b)$. Thus, $(y,T^*)\in\mathcal{Z}$, contradicting the choice of $y$. Observing that a similar argument applies in case (\ref{mins2}) and (\ref{mins3}) as well, this concludes the proof of the lemma.
\end{proof}

\begin{theorem}[multiplicity-one for generalized limit flows]\label{thm_mult_one}
Given any mean-convex initial data $K_0\subset D$, static or quasistatic density-two planes cannot occur as generalized limit flows.
\end{theorem}

\begin{proof}
Given a mean-convex initial condition $K_0\subset D$, denote by $\eps>0$ the smaller one of the two constants from Lemma \ref{lemma_isolation} (isolation) and \cite[Theorem 6.20]{EHIZ}. Consider any $\Lambda$-blowup sequence $(\widetilde{\K}^j,\widetilde{\M}^j)=(\mathcal{D}_{\lambda_j}\mcfK^j,\mathcal{D}_{\lambda_j}\mcfM^j)$, where for ease of notation we pretend that $D\subset\mathbb{R}^3_+$ and $0\in\partial\K^j$. Thanks to Proposition \ref{prop_L0} (small blowups) and the multiplicity-one theorem for blowup limits of free boundary mean-convex level set flow from \cite{EHIZ} we may assume that $\Lambda>0$.
Now, suppose towards a contradiction that $(\widetilde{\K}^j,\widetilde{\M}^j)$, converges to a limit $(\widetilde{\K}^\infty,\widetilde{\M}^\infty)$, which is a static or quasistatic density-two plane or half plane.\\

We first analyze the case where $(\widetilde{\K}^\infty,\widetilde{\M}^\infty)$ is a density-two half plane.
Let $\mu_j>0$ be the largest number such that
\begin{equation}
\phi_+(\dilD_{\mu_j}\widetilde{\K}^j) \geq \eps/2.
\end{equation}
Note that $\mu_j\to 0$. After passing to a subsequence, we can also assume that $\dilD_{\mu_i}\widetilde{\K}^j$ Hausdorff converges to a limit $\mathcal{K}$. By construction, we have
\begin{equation}\label{resc_ass}
\phi_+(\mcfK)\geq \eps/2\qquad \mathrm{ but } \qquad \phi_+(\dilD_{\lambda}\mcfK) \leq \eps/2 \,\, \mathrm{for} \,\, \lambda\geq 1.
\end{equation}

If we had $\mu_j\lambda_j /H^{j}_{\textrm{neck}}\to 0$, then by Proposition \ref{prop_L0} (small blowups) we would see that $\mcfK$ is a blowup limit or homothetic copy of the level set flow. But then, using \eqref{resc_ass} and arguing similarly as in the prof of \cite[Theorem 6.23]{EHIZ}, we could construct a blowup limit of the free boundary level set flow that is a non-planar minimal cone. Thus, after passing to a subsequence we can assume that $\mu_j\lambda_j/H^{j}_{\textrm{neck}}\to \Lambda'>0$. In particular, $\mu_j\lambda_j\to\infty$ and $\mcfK$ is a generalized limit flow that comes with a family of Radon measures $\mathcal{M}$, which is provided by Theorem \ref{prop_compactness} (hybrid compactness) or Proposition \ref{prop_LI} (large blowups), respectively.

Now, by Lemma \ref{lemma_isolation} (isolation) and \eqref{resc_ass}, any tangent flow to $(\K,\M)$ at the space-time origin must be a static or quasistatic multiplicity-two half plane. So we can apply Lemma \ref{lemma_min_surf} (minimal surface) to obtain a free boundary minimal surface $\Sigma\subset\mathbb{R}^3_+$ containing the origin and a real number $b$, such that the associated reflected quantities satisfy
\begin{equation}
\Theta(\widetilde{\M},\cdot)\geq 2\quad\mathrm{on}\quad \widetilde{\Sigma}\times (-\infty,b].
\end{equation}
In particular, it follows that $\Lambda'<\infty$. Now translate $(\K,\M)$ by $(x,t)\mapsto (x,t+j)$ and using Theorem \ref{prop_compactness} (hybrid compactness) along  $j\to\infty$ pass to a subsequential limit to get a static Brakke $\delta$-flow $(\K',\M')$ satisfying
\begin{equation}
K'_t = \bigcup_{s}K_s\quad \forall t,\qquad\textrm{and}\qquad
\Theta(\widetilde{\M'},\cdot)\geq 2\,\,\mathrm{on}\,\, \widetilde{\Sigma}\times (-\infty,\infty).
\end{equation}
Observe that since the flow is static, it does not contain any surgeries. Moreover, by \eqref{resc_ass} it is not planar. Hence, taking a tangent flow at $-\infty$ to $(\widetilde{\K'},\widetilde{\M'})$ we obtain a non-planar static minimal cone, which gives the desired contradiction.\\

Finally, if $(\widetilde{\K}^\infty,\widetilde{\M}^\infty)$ is a density-two plane, then we can obtain a similar contradiction as above, provided we now work with the more general quantity $\psi=\min(\varphi,\varphi_h)$, similarly as in the proof of \cite[Theorem 6.23]{EHIZ}. This concludes the proof of the theorem.
\end{proof}

\bigskip

\section{Partial regularity and convexity}

In this section, we show that generalized limit flows have small singular set and nonnegative second fundamental form.

\begin{theorem}[partial regularity]\label{thm_part_reg}
The singular set of any generalized limit flow has parabolic Hausdorff dimension at most $1$.
\end{theorem}

\begin{proof}
By Corollary \ref{prop_onesidedmin} (one-sided minimization for generalized limit flows) and Theorem \ref{thm_mult_one} (multiplicity-one for generalized limit flows) nontrivial cones and higher-multiplicity planes cannot occur as tangent flows of generalized limit flows. Using this, the assertion follows from standard dimension reduction and the local regularity theorem \cite{White_stratification,white_regularity,Edelen_brakke}.
\end{proof}

\begin{proposition}[rigidity]\label{prop_rig}
Let $(\mathcal{K},\mathcal{M})$ be a generalized limit flow. If $0\in\partial K_{0}$ is a regular point, and $H(0,0)=0$, then $(\mathcal{K},\mathcal{M})\cap \{t\leq 0\}$ is a flat density-one plane or half plane.
\end{proposition}

\begin{proof}
Suppose that $(0,0)$ is a regular point and $H(0,0)=0$. Then by the strict maximum principle, possibly applied to the reflected flow, we have $H\equiv 0$ in some backwards parabolic ball $P_{-}(0,\rho)$. By Theorem \ref{thm_part_reg} (partial regularity) we can choose a time $t_0\in (-\rho^2,0)$ at which the solution is completely smooth. Then, again by the strict maximum principle, there is an entire connected component $\Sigma\subset M_{t_0}$ that contains the origin and on which the mean curvature vanishes identically. Note that $\Sigma$ must be noncompact, since there are no compact minimal surfaces in Euclidean space and no compact free boundary minimal surfaces in Euclidean halfspace. Next, again by the smallness of the singular set any $X\in \Sigma\times (-\infty,t_0]\setminus\textrm{sing}\mathcal{M}$ can be connected to $(0,t_0)$ by a time-like space-time curve that entirely avoids the singular set. Together with the strict maximum principle this yields
\begin{equation}
\Sigma\times (-\infty,t_0]\subseteq \partial\mathcal{K}.
\end{equation}
In particular, there are no surgeries near $\Sigma$.
Now, by Corollary \ref{prop_onesidedmin} (one-sided minimization for generalized limit flows) and Theorem \ref{thm_mult_one} (multiplicity-one for generalized limit flows) in the case $\Lambda>0$, and by Proposition \ref{prop_L0} (convergence to level-set flow) and the one-sided minimization and multiplicity-one theorem for blowups of the level-set flow from \cite{EHIZ} in the case $\Lambda=0$, the tangent cone at infinity of $\Sigma$ or  $\widetilde{\Sigma}$, respectively, must be a multiplicity-one plane. Hence, by monotonicity, $\Sigma$ or  $\widetilde{\Sigma}$, respectively, is flat. Finally, by White's strong half space result \cite[Theorem 7]{White_nature} there cannot be any other connected components.
\end{proof}

\begin{theorem}[nonnegative second fundamental form]\label{thm_conv}
Let $(\mathcal{K},\mathcal{M})$ be a generalized limit flow. Then at every regular point all principal curvatures are nonnegative.
\end{theorem}

\begin{proof}
Fixing $K_0\subset D$, suppose towards a contradiction that there is a sequence of generalized limit flows $(\mathcal{K}^j,\mathcal{M}^j)$ and a sequence of regular points $X_j$ such that $\frac{\lambda_1}{H}(X_j)$ converges to an infimal value $\gamma< 0$. Note that $\gamma>-\infty$ thanks to the bound $|A|\leq CH$ from Proposition \ref{basic_prop} (basic properties). Moreover, by translating and scaling we may assume that $X_j=0$ and
\begin{equation}\label{rescaling_white}
\sup_{P(0,1)} |A_{\partial\widetilde{\mathcal{K}}^j}| \leq 1\leq \sup_{\overline{P(0,1)}} |A_{\partial\widetilde{\mathcal{K}}^j}|.
\end{equation}
If there is no $r>0$ such that the flow is unmodified by surgeries in $P(0,r)$, then after adjusting our sequence, using in particular item (c) of Definition \ref{def_replacestd} (replacement by standard caps or standard half caps), we may assume that $(0,0)$ lies in the presurgery domain.

If $(\mathcal{K}^j,\mathcal{M}^j)$ is a $0$-blowup sequence, then using Proposition \ref{prop_L0} (small blowups) and Theorem \ref{thm_no_microscopic} (no microscopic surgeries) we obtain contradiction with the convexity theorem for blowup limits of the free boundary level-set flow from \cite{EHIZ}. Hence, by Theorem \ref{prop_compactness} (hybrid compactness) and Proposition \ref{prop_LI} (large blowups) we may assume that $(\mathcal{K}^j,\mathcal{M}^j)$ converges to a generalized limit flow $(\mathcal{K},\mathcal{M})$. By \eqref{rescaling_white} and Proposition \ref{prop_rig} (rigidity) the limit $(\mathcal{K},\mathcal{M})$ must have strictly positive mean curvature. Hence $\lambda_1/H$ attains a strictly negative minimum at the space-time origin, which, possibly after considering the reflected flow, contradicts the strict maximum principle. This proves the theorem.
\end{proof}

\bigskip

\section{Canonical neighborhoods and existence theorem}

In this final section, we prove the canonical neighborhood theorem and the existence theorem for free boundary flows with surgery. 

\begin{theorem}[canonical neighborhoods]\label{thm_can_nbd}
Suppose $\K^j$ is a sequence of $(\delta,\mathcal{H}^j)$-flows starting at a smooth compact strictly mean-convex domain $K_0\subset D$, such that $\delta\leq\bar{\delta}$ and $\mathcal{H}^j\to \infty$. Then, for any sequence of space-time points $X_j=(x_j,t_j)\in\partial\K^j$ with  $\sup_j t_j<\infty$ and $H(X_j)\to\infty$, the rescaled flows
$\mathcal{D}_{H(X_j)}( \K^j-X_j)
$ subsequentially converge to either
\begin{itemize}
\item the evolution of a standard cap preceded by a round shrinking cylinder, or a round shrinking cylinder, round shrinking sphere, translating bowl or ancient oval, or
\item the evolution of a standard half cap preceded by a round shrinking half cylinder, or a round shrinking half cylinder, round shrinking half sphere, translating half bowl or ancient half oval.
\end{itemize}
\end{theorem}

\begin{proof}
Consider the blowup sequence $\widetilde{\K}^j:=\mathcal{D}_{\lambda_j}( \K^j-X_j)$, where $\lambda_j$ is chosen such that
\begin{equation}\label{rescaling_white_again}
\sup_{P(0,1)} |A_{\partial\widetilde{\mathcal{K}}^j}| \leq 1\leq \sup_{\overline{P(0,1)}} |A_{\partial\widetilde{\mathcal{K}}^j}|.
\end{equation}
In light of Proposition \ref{prop_L0} (small blowups) and Proposition \ref{prop_LI} (large blowups) we may assume that
$\lambda_j/H^j_{\textrm{neck}}\to \Lambda \in (0,\infty)$,
since \cite[Corollary 1.3]{EHIZ} or the argument from its proof, respectively, already yields the conclusion in the other cases. Hence, by Theorem \ref{prop_compactness} (hybrid compactness), considering the associated family of Radon measures $\widetilde{\M}^j$ defined as in \eqref{ass_radon}, we can pass to a subsequential limit $(\K,\M)$.
Define $\mathcal{K}^0=\{K_t^0\}_{t\leq 0}$ by for each $t\leq 0$ setting $K_t^0$ to be the connected component of $K_t$ that contains the origin. Note that all these sets in fact contain a closed ball $B$ of positive radius thanks to \eqref{rescaling_white_again}. By Theorem \ref{thm_part_reg} (partial regularity) together with Theorem \ref{thm_conv} (nonnegative second fundamental form) and connectedness, the sets $K_t^0$ are smooth and convex for almost every $t$. Remembering in particular the way surgeries are performed, we see that convexity in fact holds at all $t\leq 0$.
Now given any $p\in \partial K_t^0$ the convex hull of $p$ and $B$ is contained in $K_t^0$, and consequently, remembering Corollary \ref{prop_onesidedmin} (one-sided minimization for generalized limit flows) and Theorem \ref{thm_mult_one} (multiplicity-one for generalized limit flows), any tangent flow at $(p,t)$ must be a density-one plane or half plane. Hence, $K_t^0$ is smooth and convex for all $t\leq 0$.\\

If there are no surgeries, then by the classification from \cite{BC,ADS}, the flow $\mathcal{K}^0$  or its reflection $\widetilde{\mathcal{K}^0}$, respectively, must be a round shrinking cylinder, round shrinking sphere, translating bowl or ancient oval. Assume now $\mathcal{K}^0$ does contain a surgery, let $T\leq 0$ be a surgery time and let $N\subset {K}_T^0$ be a surgery neck or half neck of quality $\bar{\delta}$ sitting in the backward time slice.
Note that ${N}$ is the limit of some $\widetilde{N}^j$ in the approximators $\widetilde{\K}^j$.
By part (b) of Definition \ref{def_MCF_surgery} (free boundary flow with surgery) we can find a curve $\gamma_j$ in the approximator connecting
$\{H=H_{\textrm{trig}}^j\}$ and $\{H\leq H_{\textrm{th}}^j\}$,  such that it passes through $\widetilde{N}^j$ but avoids all other $\bar{\delta}$-necks and half $\bar{\delta}$-necks of the disjoint collection. We can assume that the curve $\gamma_j$ enters and leaves $\widetilde{N}^j$ exactly once. 
Let $p_j$ be the center of $\widetilde{N}^j$. Since $\mathcal{K}^0$ is smooth with strictly positive mean curvature, and since $\mathcal{H}^j\rightarrow \infty$, given any $R<\infty$, for $j$ large enough the curve $\gamma_j$ must start and end outside $B(p_j, R )$. Thus, ${K}_T^{0,-}\setminus N$ has at least two unbounded components. Since ${K}_T^{0,-}$ is connected, ${K}_T^{0,-}\setminus N$ must have exactly two components.
We have thus shown that ${K}_T^{0,-}$ has two ends, and consequently it contains a line, and all prior time slices contain this line as well.
Hence, at each fixed time the convex set splits off an $\R$-factor, and thus there cannot be any other surgeries. It follows that ${\K}^0$ or $\widetilde{\mathcal{K}^0}$, respectively, is a round cylindrical flow for $t<T$ followed by the unique evolution of the standard cap for $t>T$.\\

Finally, since all the ancient solutions from the above sweep out the entire space for $t\to -\infty$ it follows that there are in fact no other connected components, i.e. $\K^0=\K$, and since $H(X_j)\to H(0)\in (0,\infty)$, remembering also Theorem \ref{thm_no_microscopic} (no microscopic surgeries), we can rescale by $H(X_j)$ instead of $\lambda_j$ to conclude the proof of the theorem.
\end{proof}

We can now prove our main theorem, which we restate here for convenience of the reader:

\begin{theorem}[free boundary flow with surgery]
Given any smooth compact strictly mean-convex free boundary domain $K_0\subset D$, for suitable choice of the surgery parameters $\delta$ and $\mathcal{H}$, there exists a free boundary $(\delta,\mathcal{H})$-flow $\{K_t\}_{t\in[0,\infty)}$  starting at $K_0$. Moreover, the flow either becomes extinct in finite time or for $t\to\infty$ converges smoothly in the one or two-sheeted sense to a finite collection of stable connected minimal surfaces with empty or free boundary (in particular, there are no surgeries for $t$ sufficiently large).
\end{theorem}

\begin{proof}
Given any $T<\infty$, we will first prove the existence of a free boundary $(\delta,\mathcal{H})$-flow $\{K_t\}$ starting at $K_0$ and defined on the finite interval $[0,T]$ via a continuity argument similarly as in \cite[Section 4.2]{HK}. 
To this end, fixing $\bar{\delta}>0$ small enough, suppose towards a contradiction that there is a sequence $\K^j$ of free boundary $(\delta,\mathcal{H}_j)$-flows with $\delta\leq \bar{\delta}$ and $\mathcal{H}^j\to\infty$, that can only be defined on a maximal time interval $[0,T_j]$ for some
$T_j<T$.
Then, it must be the case that we cannot find a minimal collection
of strong $\delta$-necks and strong half $\delta$-necks in $K^j_{T_j}$
as required in Definition \ref{def_MCF_surgery} (free boundary flow with surgery), since otherwise we could perform surgeries along an `innermost' such collection of centers $p$, i.e. one for which
$
\sum_p \textrm{dist}(p,\{H=H_{\textrm{trig}}^j\})
$
is minimal,  and run smooth free boundary mean curvature flow for a short time, contradicting the maximality of $T_j$.
So our goal is to  
produce a minimal separating collection of strong $\delta$-necks and strong half $\delta$-necks for large $j$.

Let $\mathcal{I}_j$ be the set of points $p\in \partial K^j_{T_j}$ with $H(p)>H^j_{\textrm{neck}}$, and let $\mathcal{J}_j$ be the set of points $p\in \partial  K^j_{T_j}$ with $H(p)=H^j_{\textrm{neck}}$. Then, similarly as in \cite[Claim 4.6]{HK}
there is a large constant $C<\infty$, such that the union $
V_j=\bigcup_{p\in \mathcal{J}_j} B(p,CH^{-1}(p))$, for $j$ large enough, separates $\{H=H^j_{\textrm{trig}}\}$ from $\{H\leq H^j_{\textrm{th}}\}$ in the domain $K^j_{T_j}$.
Let $\hat{\mathcal{J}}_j\subseteq\mathcal{J}_j$ be a minimal subset
such that the union of balls $\cup_{p\in \hat{\mathcal{J}}_j}B(p,CH^{-1}(p))$ has the separation property.
Then, using Theorem \ref{thm_can_nbd} (canonical neighborhoods) and arguing similarly as in the proof of \cite[Claim 4.7]{HK} we see that
 for large $j$, every $p\in \hat{\mathcal{J}}_j$ is a strong $\delta$-neck point or strong half $\delta$-neck point.  This collection is disjoint for large $j$, since otherwise two intersecting $\delta$-necks or half $\delta$-necks would lie in a single single one with quality $\hat{\delta}\ll\delta$, which is impossible by minimality of $\mathcal{J}_j$.
Thus, we have a minimal collection of disjoint strong $\delta$-necks and strong half $\delta$-necks with the separation property; this gives the desired contradiction and thus proves the existence on the interval $[0,T]$.

Finally, it is known by \cite[Theorem 1.5]{EHIZ} that the free boundary level set flow of $K_0$ either becomes extinct in finite time or or for $t\to\infty$ converges smoothly in the one or two-sheeted sense to a finite collection of stable minimal surfaces or stable free boundary minimal surfaces. Hence, applying the above result for $T<\infty$ sufficiently large and taking also into account Proposition \ref{prop_levelset} (distance to free boundary level set flow), by mimicking the argument from Brendle-Huisken \cite{BrendleHuisken_ambient}, we can get long-time existence and convergence. 
\end{proof}

\bigskip

\bibliography{fb_surgery}

\bibliographystyle{alpha}

\vspace{10mm}

{\sc Robert Haslhofer, Department of Mathematics, University of Toronto,  40 St George Street, Toronto, ON M5S 2E4, Canada}\\

\emph{E-mail:} roberth@math.toronto.edu

\end{document}